\documentclass[12pt]{amsart}
\usepackage{amsmath,amsthm,latexsym,amscd,amsbsy,amssymb,amsfonts,fleqn,leqno,
euscript, pb-diagram}

\numberwithin{equation}{section}

\newcommand{\I}{\mathbb I}

\newcommand{\reg}{\mathrm{e}}

\newtheorem{thm}{Theorem}[section]
\newtheorem{pro}[thm]{Proposition}
\newtheorem{lem}[thm]{Lemma}
\newtheorem{cor}[thm]{Corollary}

\begin{document}


\title[External characterization of I-favorable spaces]
{External characterization of I-favorable spaces}

\author{Vesko  Valov}
\address{Department of Computer Science and Mathematics, Nipissing University,
100 College Drive, P.O. Box 5002, North Bay, ON, P1B 8L7, Canada}
\email{veskov@nipissingu.ca}
\thanks{Research supported in part by NSERC Grant 261914-08}

\keywords{compact spaces, continuous inverse systems, $\mathrm
I$-favorable spaces, skeletal maps}

\subjclass{Primary 54C10; Secondary 54F65}


\begin{abstract}
We provide both a spectral and an internal characterizations of
arbitrary $\mathrm{I}$-favorable spaces with respect to co-zero
sets. As a corollary we establish that any product of compact
$\mathrm{I}$-favorable spaces with respect to co-zero sets is also
$\mathrm{I}$-favorable with respect to co-zero sets.  We also prove
that every $C^*$-embedded $\mathrm{I}$-favorable with respect to
co-zero sets subspace of an extremally disconnected space is
extremally disconnected.
\end{abstract}

\maketitle

\markboth{}{I-favorable spaces}



\section{Introduction}
In this paper we assume that the topological spaces are Tychonoff
and the single-valued maps are continuous. Moreover, all inverse
systems are supposed to have surjective bonding maps.

P. Daniels, K. Kunen and H. Zhou \cite{dkz} introduced the so called
open-open game between two players, and the spaces with a winning
strategy for the first player were called $\mathrm{I}$-favorable.
Recently A. Kucharski and S. Plewik (see \cite{kp1}, \cite{kp2} and
\cite{kp3}) investigated the connection of $\mathrm{I}$-favorable
spaces and skeletal maps. In particular, they proved in \cite{kp2}
that the class of compact $\mathrm{I}$-favorable spaces and the
skeletal maps are adequate in the sense of E. Shchepin \cite{sc1}.

On the other hand, the author announced \cite[Theorem 3.1(iii)]{v} a
characterization of the class of spaces admitting a lattice
\cite{sc1} of skeletal maps (the skeletal maps in \cite{v} were
called $\mathrm{ad}$-open maps) as dense subset of the limit spaces
of $\sigma$-complete almost continuous inverse systems with skeletal
projections. Moreover, an internal characterization of the above
class was also announced \cite[Theorem 3.1(ii)]{v}. In this paper we
are going to show that the later class coincides with that one of
{\em $\mathrm{I}$-favorable spaces with respect to co-zero sets},
and to provide the proof of these characterizations. Therefore, we
obtain both a spectral and an internal characterizations of
$\mathrm{I}$-favorable spaces with respect to co-zero sets.

The following theorem is our main result:

\begin{thm}
For a space $X$ the following conditions are equivalent:
\begin{itemize}
\item[(i)] $X$ is $\mathrm{I}$-favorable with respect to co-zero sets;
\item[(ii)] Every $C^*$-embedding of $X$ in another space
is $\pi$-regular;
\item[(iii)] $X$ is skeletally generated.
\end{itemize}
\end{thm}

We say that a subspace $X\subset Y$ is {\em $\pi$-regularly
embedded} in $Y$ \cite{v} if there exists a $\pi$-base $\mathcal B$
for $X$ and a function $\reg\colon\mathcal B\to\mathcal T_Y,$ where
$\mathcal T_Y$ is the topology of $Y$, such that:
\begin{itemize}
\item[(1)] $\reg(U)\cap X$ is a dense subset of $U$;
\item[(2)] $\reg(U)\cap\reg(V)=\varnothing$ provided $U\cap
V=\varnothing$.
\end{itemize}
It is easily seen that the above definition doesn't change if
$\mathcal B$ is either a base for $X$ or $\mathcal B=\mathcal T_X$.

A space $X$ is {\em skeletally generated} if there exists an inverse
system $\displaystyle S=\{X_\alpha, p^{\beta}_\alpha, A\}$ of
separable metric spaces $X_\alpha$ such that:
\begin{itemize}
\item[(3)] All bonding maps $p^{\beta}_\alpha$ are surjective and
skeletal;
\item[(4)] The index set $A$ is $\sigma$-complete (every countable chain in
$A$ has a supremum in $A$);
\item[(5)] For every countable chain $\{\alpha_n:n\geq 1\}\subset A$ with
$\beta=\sup\{\alpha_n:n\geq 1\}$ the space $X_\beta$ is a (dense)
subset of
$\displaystyle\underleftarrow{\lim}\{X_{\alpha_n},p^{\alpha_{n+1}}_{\alpha_n}\}$;
\item[(6)] $X$ is embedded in $\displaystyle\underleftarrow{\lim}
S$ such that $p_\alpha(X)=X_\alpha$ for each $\alpha$, where
$p_\alpha\colon\displaystyle\underleftarrow{\lim}S\to X_\alpha$ is
the $\alpha$-th limit projection;
\item[(7)] For every bounded continuous function $f\colon X\to\mathbb R$
there exists $\alpha\in A$ and a continuous function $g\colon
X_\alpha\to\mathbb R$ with $f=g\circ (p_\alpha|X)$.
\end{itemize}
We say that an inverse system $S$ satisfying conditions $(3) - (6)$
is {\em almost $\sigma$-continuous}. Let us note that condition
$(6)$ implies that $X$ is a dense subset of
$\displaystyle\underleftarrow{\lim}S$.

There exists a similarity between $\mathrm{I}$-favorable spaces with
respect to co-zero sets and $\kappa$-metrizable compacta \cite{sc2}.
Item $(ii)$ is analogical to Shirokov's \cite{s} external
characterization of $\kappa$-metrizable compacta, while the
definition of skeletally generated spaces resembles that one of
openly generated compacta \cite{sc3}. Moreover, according to
Shapiro's result \cite{s}, every continuous image of a
$\kappa$-metrizable compactum is skeletally generated, so it is
$\mathrm{I}$-favorable with respect to co-zero sets. So, next
question seems reasonable.

\smallskip\noindent
{\bf Question.} Is there any characterization of $\kappa$-metrizable
compacta in terms of a game between two players?

\smallskip

It is shown in \cite[Corollary 1.7]{dkz} that the product of
$\mathrm{I}$-favorable spaces is also $\mathrm{I}$-favorable. Next
corollary shows that a similar result is true for
$\mathrm{I}$-favorable spaces with respect to co-zero sets.

\begin{cor}
Any product of compact $\mathrm I$-favorable spaces with respect to
co-zero sets is also $\mathrm I$-favorable with respect to co-zero
sets.
\end{cor}


Corollary 1.3 below is similar to a result of Bereznicki\v{i}
\cite{b} about specially embedded subset of extremally disconnected
spaces.

\begin{cor}
Let $X$ be a $C^*$-embedded subset of an extremally disconnected
space. If $X$ is $\mathrm I$-favorable with respect to co-zero sets,
then it is also extremally disconnected.
\end{cor}

\section{$\mathrm{I}$-favorable spaces with respect to co-zero sets}

In this section we consider a modification of the open-open game
when the players are choosing co-zero sets only. Let us describe
this game. Players are playing in a topological space $X$. Player I
choose a non-empty co-zero set $A_0\subset X$, then Player II choose
a non-empty co-zero set $B_0\subset A_0$. At the $n$-th round Player
I choose a non-empty co-zero set $A_n\subset X$ and the Player II is
replying by choosing a non-empty co-zero set $B_n\subset A_n$.
Player I wins if the union $B_0\cup B_1\cup...$ is dense in $X$,
otherwise Player II wins. The space $X$ is called {\em
$\mathrm{I}$-favorable with respect to co-zero sets} if Player I has
a winning strategy. Denote by $\Sigma_X$ the family of all non-empty
co-zero sets in $X$. A winning strategy, see \cite{kp1}, is a
function $\sigma:\bigcup\{\Sigma_X^n:n\geq 0\}\to\Sigma_X$ such that
for each game
$$\big(\sigma(\varnothing),B_0,\sigma(B_0),B_1,\sigma(B_0,B_1),B_2,...,B_n,\sigma(B_0,B_1,..,B_n),B_{n+1},,,\big),$$
where $B_k$ and $\sigma(\varnothing)$ belong to $\Sigma_X$ and
$B_{k+1}\subset\sigma(B_0,B_1,..,B_k)$ for every $k\geq 0$, the
union $\bigcup_{n\geq 0}B_n$ is dense in $X$.
For example, every space with a countable $\pi$-base $\mathcal B$ of
co-zero sets is $\mathrm{I}$-favorable with respect to co-zero sets
(the strategy for Player I is to keep choosing every member of
$\mathcal B$, see \cite[Theorem 1.1]{dkz}). Let us mention that if
in the above game the players are choosing arbitrary open subsets of
$X$ and Player I has a winning strategy, then $X$ is called $\mathrm
I$-favorable, see \cite{dkz}.

\begin{pro}
If $X$ is $\mathrm{I}$-favorable with respect to co-zero sets, so is
$\beta X$.
\end{pro}

\begin{proof}
Let $\sigma:\bigcup\{\Sigma_X^n:n\geq 0\}\to\Sigma_X$ be a winning
strategy for Player I. Observe that for every co-zero set $U$ in $X$
there exists a co-zero set $c(U)$ in $\beta X$ with $c(U)\cap X=U$.
Now define a function $\overline{\sigma}:\bigcup\{\Sigma_{\beta
X}^n:n\geq 0\}\to\Sigma_{\beta X}$ by
$$\overline{\sigma}(U_1,..,U_n)=c\big(\sigma(U_1\cap
X,..,U_n\cap X)\big).$$ Suppose
$$\big(\overline{\sigma}(\varnothing),U_0,\overline{\sigma}(U_0),U_1,\overline{\sigma}(U_0,U_1),...,U_n,\overline{\sigma}(U_0,U_1,..,U_n),U_{n+1},,,\big)$$
is a sequence such that $\overline{\sigma}(\varnothing)$ and all
$U_k$ belong to $\Sigma_{\beta X}$ with
$U_{k+1}\subset\overline{\sigma}(U_0,U_1,..,U_k)$ for each $k\geq
0$. Consequently, $U_{k+1}\cap X\subset\sigma(U_0\cap X,..,U_k\cap
X)$, $k\geq 0$. So, the set $X\cap\bigcup_{k\geq 0}U_k$ is dense in
$X$ which implies that $\bigcup_{k\geq 0}U_k$ is dense $\beta X$.
Therefore, $\beta X$ is $\mathrm{I}$-favorable with respect to
co-zero sets.
\end{proof}

A map $f\colon X\to Y$ is said to be skeletal if the closure
$\overline{f(U)}$ of $f(U)$ in $Y$ has a non-empty interior in $Y$
for every open set $U\subset X$. The proof of next lemma is
standard.

\begin{lem} For a map $f\colon X\to Y$ the following are equivalent:
\begin{itemize}
\item[(i)] $f$ is skeletal;
\item[(ii)] $\overline{f(U)}$ is regularly closed in $Y$, i.e., its
interior $\mathrm{Int}\overline{f(U)}$ in $Y$ is dense in
$\overline{f(U)}$ for every open $U\subset X$;
\item[(iii)] Every open $U\subset X$ contains an open set $V_U$ such
that $f(V_U)$ is dense in some open subset of $Y$.
\end{itemize}
If in addition $f$ is closed, the above three conditions are
equivalent to $f(U)$ has a non-empty interior in $Y$ for every open
$U\subset X$.
\end{lem}

A space $X$ is said to be an {\em almost limit} of the inverse
system $\displaystyle S=\{X_\alpha, p^{\beta}_\alpha, A\}$ if $X$
can be embedded in $\displaystyle\underleftarrow{\lim}S$ such that
$p_\alpha(X)=X_\alpha$ for each $\alpha$. We denote this by
$X=\displaystyle\mathrm{a}-\underleftarrow{\lim}S$, and it implies
that $X$ is a dense subset of $\displaystyle\underleftarrow{\lim}
S$. Let $\displaystyle S=\{X_\alpha, p^{\beta}_\alpha,
\alpha<\beta<\tau\}$ be a well ordered inverse system with
(surjective) bonding maps $p^{\beta}_\alpha$, where $\tau$ is a
given cardinal. We say that $S$ is {\em almost continuous} if for
every limit cardinal $\gamma<\tau$ the space $X_\gamma$ is naturally
embedded in the limit space
$\displaystyle\underleftarrow{\lim}\{X_\alpha, p^{\beta}_\alpha,
\alpha<\beta<\gamma\}$. If always
$X_\gamma=\displaystyle\underleftarrow{\lim}\{X_\alpha,
p^{\beta}_\alpha, \alpha<\beta<\gamma\}$, $S$ is called {\em
continuous}.

\begin{lem}
Let $X=\displaystyle\mathrm{a}-\underleftarrow{\lim}\{X_\alpha,
p^{\beta}_\alpha, A\}$ such that all bonding maps $\displaystyle
p^{\beta}_\alpha$ are skeletal. Then all $p_\alpha$ and the
restrictions $p_\alpha|X\colon X\to X_\alpha$ are also skeletal.
\end{lem}

\begin{proof}
Since $X$ is dense in $\displaystyle\underleftarrow{\lim}\{X_\alpha,
p^{\beta}_\alpha, A\}$, $p_\alpha$ is skeletal iff so is
$p_\alpha|X$, $\alpha\in A$. To prove that a given $p_\alpha$ is
skeletal, let $U\subset\displaystyle\underleftarrow{\lim}\{X_\alpha,
p^{\beta}_\alpha, A\}$ be an open set. We are going to show that
$\mathrm{Int}\overline{p_\alpha(U)}\neq\varnothing$ (both, the
interior and the closure are in $X_\alpha$). We can suppose that
$U=p_\beta^{-1}(V)$ for some $\beta$ with $V\subset X_\beta$ being
open. Moreover, since $A$ is directed, there exists $\gamma\in A$
with $\beta<\gamma$ and $\alpha<\gamma$. Then,
$p_\alpha(U)=p^\gamma_\alpha(W)$, where
$W=(p^\gamma_\beta)^{-1}(V)$. Finally, because $p^\gamma_\alpha$ is
skeletal, $\mathrm{Int}\overline{p_\alpha(U)}\neq\varnothing$.
\end{proof}

\begin{lem}
Every skeletally generated space is $\mathrm{I}$-favorable with
respect to co-zero sets.
\end{lem}

\begin{proof}
Let $X=\displaystyle\mathrm{a}-\underleftarrow{\lim}S$, where
$S=\{X_\alpha, p^{\beta}_\alpha, A\}$ satisfies conditions (3)-(7).
Condition (7) implies that for every co-zero set $U\subset X$ there
exists $\alpha\in A$ and a co-zero set $V\subset X_\alpha$ with
$U=p_\alpha^{-1}(V)$. So, $\Sigma_X$ is the family of all
$p_\alpha^{-1}(V)$, where $\alpha\in A$ and $V$ is open in
$X_\alpha$. Using this observation, we can apply the arguments from
the proof of \cite[Theorem 2]{kp3} to define a winning strategy
$\sigma:\bigcup\{\Sigma_X^n:n\geq 0\}\to\Sigma_X$.
\end{proof}

We are going to show that every compactum $X$ which is
$\mathrm{I}$-favorable with respect to co-zero sets can be
represented as a limit of a continuous system with skeletal bonding
maps and $\mathrm{I}$-favorable spaces with respect to co-zero sets
of weight less than the weight $w(X)$ of $X$.

Let us introduced few notations. Suppose $X\subset\mathbb I^A$ is a
compact space and $B\subset A$. Let $\pi_B\colon\mathbb
I^A\to\mathbb I^B$ be the natural projection and $p_B$ be
restriction map $\pi_B|X$. Let also $X_B=p_B(X)$. If $U\subset X$ we
write $B\in k(U)$ to denote that $p_{B}^{-1}\big(p_{B}(U)\big)=U$.
For every co-zero set $U\subset X$ there exist a countable $B\subset
A$ such that $B\in k(U)$ with $p(U)$ being a co-zero set in $X_{B}$.
A base $\mathcal B$ for the topology of $X\subset\mathbb I^A$
consisting of co-zero sets is called {\em special} if for every
finite $B\subset A$ the family $\{p_B(U):U\in\mathcal B, B\in
k(U)\}$ is a base for $p_B(X)$.

\begin{pro}
Let $X\subset\mathbb I^A$ be a compactum and $\mathcal B$ a special
base for $X$. If $\sigma:\bigcup\{\mathcal B^n:n\geq 0\}\to\mathcal
B$ is a function such that for each game
$$\big(\sigma(\varnothing),U_0,\sigma(U_0),U_1,\sigma(U_0,U_1),U_2,...,U_n,\sigma(U_0,U_1,..,U_n),U_{n+1},,,\big),$$
where $\sigma(\varnothing)\in\mathcal B$, $U_i\in\mathcal B$ and
$U_{i+1}\subset\sigma(U_0,U_1),U_2,...,U_i)$  for all $i\geq 0$, the
union $\bigcup_{n\geq 0}U_n$ is dense in $X$, then $X$ is skeletally
generated.
\end{pro}

\begin{proof}
For any finite set $B\subset A$ fix a countable family
$\lambda_B\subset\mathcal B$ such that $\{p_B(U):U\in\lambda_B\}$ is
a base for $X_B$ and $B\in k(U)$ for every $U\in\lambda_B$. Let
$\gamma_B=\bigcup\{\lambda_H:H\subset B\}$ and $\Gamma$ be the
family of all countable sets $B\subset A$ satisfying the following
condition:
\begin{itemize}
\item If $C\subset B$ is finite and $U_0,U_1,...,U_n\in\gamma_C$, $n\geq 0$, then
$B\in k\big(\sigma(U_0,U_1,..,U_n)\big)$.
\end{itemize}
Obviously, if $B_1\subset B_2\subset..$ is a chain in $\Gamma$, then
$\bigcup_{i\geq 1}B_i\in\Gamma$. We claim that
$X=\displaystyle\underleftarrow{\lim}\{X_B, p^C_B, B\subset C,
\Gamma\}$. It suffices to show that every countable subset of $A$ is
contained in an element of $\Gamma$. To this end, let $B_0\subset A$
be countable. Construct by induction countable sets $B(m)\subset A$
such that for all $m\geq 0$ we have:
\begin{itemize}
\item $B_0\subset B(m)\subset B(m+1)$;
\item $B(m+1)\in k\big(\sigma(U_0,U_1,..,U_n)\big)$, where $U_0,U_1,..,U_n\in\gamma_C$ with
$n\geq 0$ and $C\subset B(m)$ finite.
\end{itemize}

Suppose $B(j)$, $j\leq m$, are already constructed for some $m\geq
1$. For every finite $C\subset B(m)$ and
$U_0,U_1,..,U_n\in\gamma_{C}$ there exist a countable set
$B(U_0,U_1,..,U_n)\subset A$ with $B(U_0,U_1,..,U_n)\in
k\big(\sigma(U_0,U_1,..,U_n)\big)$. Let $B(m+1)$ be the union of
$B(m)$ and all $B(U_0,U_1,..,U_n)$, where
$U_0,U_1,..,U_n\in\gamma_C$ with $C$ being a finite subset of $B(m)$
and $n\geq 0$. Obviously $B(m+1)$ is countable and satisfies the
required conditions. This completes the inductive step. Finally,
$B_\infty=\cup_{m=0}^\infty B(m)$ belongs to $\Gamma$. Hence,
$X=\displaystyle\underleftarrow{\lim}\{X_B, p^C_B, B\subset C,
\Gamma\}$.

Next two claims complete the proof of Proposition 2.5.

\smallskip\noindent
{\em Claim $1.$ If $B\in\Gamma$, then for each open $V\subset X$
there exists a finite set $C\subset B$ and a finite family
$U_0,U_1,..,U_n\in\gamma_C$ such that $p_B(U)\cap
p_B(V)\neq\varnothing$ for any $U\in\gamma_H$, where $H\subset B$ is
finite and $U\subset\sigma(U_0,U_1,..,U_n)$. }
\smallskip

Assume Claim 1 does not hold. Then there exists an open set
$V\subset X$ such that for any finite $C\subset B$ and any
$U_0,U_1,..,U_n\in\gamma_C$ there exists finite $H\subset B$ and
$U\in\gamma_H$ such that $U\subset\sigma(U_0,U_1,..,U_n)$ and
$p_B(U)\cap p_B(V)=\varnothing$. This allows us to construct by
induction a sequence $\{C(m)\}_{m\geq 0}$ of finite subsets of $B$
and families $\{U_0,U_1,..,U_m\}\subset\gamma_{C(m)}$ such that
$U_{m}\subset\sigma(U_0,U_1,..,U_{m-1})$ and $p_B(U_{m})\cap
p_B(V)=\varnothing$. Indeed, we take $\sigma(\varnothing)\in\mathcal
B$ with $B\in k(\sigma(\varnothing))$ and suppose the sets
$C(1),...,C(m)$ and the families
$\{U_0,U_1,..,U_m\}\subset\gamma_{C(m)}$ satisfying the above
conditions are already constructed. Consequently, there exists
$U_{m+1}\in\gamma_D$, where $D\subset B$ is finite, such that
$U_{m+1}\subset \sigma(U_0,U_1,..,U_m)$ and $p_B(U_{m+1})\cap
p_B(V)=\varnothing$. Observe that both
$\{U_0,U_1,..,U_m\}\subset\gamma_{C(m)}$ and $U_{m+1}\in\gamma_D$
implies the inclusion $\{U_0,U_1,..,U_m,
U_{m+1}\}\subset\gamma_{C(m+1)}$, where $C(m+1)=C(m)\cup D$. This
completes the inductive step. So, we obtained a sequence
$$\sigma(\varnothing),U_0,\sigma(U_0),U_1,\sigma(U_0,U_1),U_2,...,U_n,\sigma(U_0,U_1,..,U_n),U_{n+1},..$$
from $\mathcal B$ such that
$U_{i+1}\subset\sigma(U_0,U_1,U_2,...,U_i)$, $B\in k(U_i)$ and
$p_B(U_i)\cap p_B(V)=\varnothing$ for all $i$. The last two
conditions yields $U_i\cap V=\varnothing$ for all $i\ge 0$ which
contradicts the density of the set $\bigcup_{i\geq 0}U_i$ in $X$.

\smallskip\noindent
{\em Claim $2.$ $p_B$ is a skeletal map for each $B\in\Gamma$.}
\smallskip

Suppose $V\subset X$ is open. Then there a finite set $C\subset B$
and a family $U_0,U_1,..,U_n\in\gamma_C$ satisfying the conditions
from Claim 1. Since $B\in k\big(\sigma(U_0,U_1,..,U_m)\big)$,
$p_B\big(\sigma(U_0,U_1,..,U_m)\big)$ is open in $X_B$. Hence, it
suffices to show the inclusion
$p_B\big(\sigma(U_0,U_1,..,U_m)\big)\subset\overline{p_B(V)}$.
Assuming the contrary, we obtain that
$p_B\big(\sigma(U_0,U_1,..,U_m)\big)\backslash\overline{p_B(V)}$ is
a non-empty open subset of $X_B$. Moreover,
$\bigcup\{p_B(\gamma_C):C\subset B{~}\mbox{is finite}\}$ is a base
for $X_B$. Therefore, there is $U\in\gamma_C$ with $C\subset B$
finite such that $p_B(U)$ is contained in
$p_B\big(\sigma(U_0,U_1,..,U_m)\big)\backslash\overline{p_B(V)}$.
Consequently, $U\subset\sigma(U_0,U_1,..,U_m)$ and $p_B(U)\cap
p_B(V)=\varnothing$, a contradiction.
\end{proof}

\begin{thm}
Let $X$ be a compact $\mathrm{I}$-favorable space with respect to
co-zero sets and $w(X)=\tau$ is uncountable. Then there exists a
continuous inverse system $S=\{X_\alpha, p^{\beta}_\alpha, \tau\}$
of compact $\mathrm{I}$-favorable spaces $X_\alpha$ with respect to
co-zero sets and skeletal bonding maps $p^{\beta}_\alpha$ such that
$w(X_\alpha)<\tau$ for each $\alpha<\tau$ and
$X=\displaystyle\underleftarrow{\lim}S$.
\end{thm}

\begin{proof}
Let $\sigma:\bigcup\{\Sigma_X^n:n\geq 0\}\to\Sigma_X$, where
$\Sigma_X$ is the family of all co-zero sets in $X$, be a winning
strategy for Player I. We embed $X$ in a Tychonoff cube $\mathbb
I^A$ with $|A|=\tau$ and fix a base $\{U_\alpha:\alpha<\tau\}$ for
$X$ of cardinality $\tau$ which consists of co-zero sets such that
for each $\alpha$ there exists a finite set $H_\alpha$ with
$H_\alpha\in k(U_\alpha)$. For any finite set $C\subset A$ let
$\gamma_C$ be a fixed countable base for $X_C$. Observe that for
every $U\in\Sigma_X$ there exists a countable set $B(U)\subset A$
such that $B(U)\in k(U)$ and $p_{B(U)}(U)$ is a co-zero set in
$X_{B(U)}$. This follows from the fact that each continuous function
$f$ on $X$ can be represented in the form $f=g\circ p_B$ with
$B\subset A$ countable and $g$ being a continuous function on $X_B$.
We identify $A$ with all infinite cardinals $\alpha<\tau$ and
construct by transfinite induction subsets $A(\alpha)\subset A$ and
families $\mathcal U(\alpha)\subset\Sigma_X$ satisfying the
following conditions:
\begin{itemize}
\item[(8)] $|A(\alpha)|\leq\alpha$ and $|\mathcal
U(\alpha)|\leq\alpha$;
\item[(9)]  $A(\alpha)\in k(U)$ for all $U\in\mathcal
U(\alpha)$;
\item[(10)] $p_C^{-1}(\gamma_C)\subset\mathcal
U(\alpha)$ for  each finite $C\subset A(\alpha)$;
\item[(11)]  $\{U_\beta:\beta<\alpha\}\subset\mathcal
U(\alpha)$ and $\{\beta:\beta<\alpha\}\subset A(\alpha)$;
\item[(12)]  $\sigma(U_1,..,U_n)\in\mathcal
U(\alpha)$ for every finite family $\{U_1,..,U_n\}\subset\mathcal
U(\alpha)$;
\item[(13)]  $A(\alpha)=\bigcup\{A(\beta):\beta<\alpha\}$ and $\mathcal
U(\alpha)=\bigcup\{\mathcal U(\beta):\beta<\alpha\}$ for all limit
cardinals $\alpha$.
\end{itemize}
Suppose all $A(\beta)$ and $\mathcal U(\beta)$, $\beta<\alpha$, have
already been constructed for some $\alpha<\tau$. If $\alpha$ is a
limit cardinal, we put $A(\alpha)=\bigcup\{A(\beta):\beta<\alpha\}$
and $\mathcal U(\alpha)=\bigcup\{\mathcal U(\beta):\beta<\alpha\}$.
If $\alpha=\beta+1$, we construct by induction a sequence
$\{C(m)\}_{m\geq 0}$ of subsets of $A$, and a sequence $\{\mathcal
V_m\}_{m\geq 0}$ of co-zero families in $X$ such that:
\begin{itemize}
\item $C_0=A(\beta)\cup\{\beta\}$ and $\mathcal V_0=\mathcal U(\beta)\cup\{U_\beta\}$;
\item $C(m+1)=C(m)\bigcup\{B(U):U\in\mathcal V_m\}$;
\item $\mathcal V_{2m+1}=\mathcal V_{2m}\bigcup\{\sigma(U_1,..,U_s): U_1,..,U_s\in\mathcal
V_{2m}, s\geq 1\}$;
\item $\mathcal V_{2m+2}=\mathcal
V_{2m+1}\bigcup\{p_C^{-1}(\gamma_C):C\subset C(2m+1){~}\mbox{is
finite}\}$.
\end{itemize}

Now, we define $A(\alpha)=\bigcup_{m\geq 0}C(m)$ and $\mathcal
U(\alpha)=\bigcup_{m\geq 0}\mathcal V_m$. It is easily seen that
$A(\alpha)$ and $\mathcal U(\alpha)$ satisfy conditions (8)-(13).

For every $\alpha<\tau$ let $X_\alpha=X_{A(\alpha)}$ and
$p_\alpha=p_{A(\alpha)}$.  Moreover, if $\alpha<\beta$, we have
$A(\alpha)\subset A(\beta)$. In such a situation let
$p^\beta_\alpha=p^{A(\beta)}_{A(\alpha)}$. Since
$A=\bigcup_{\alpha<\tau}A(\alpha)$, we obtain a continuous inverse
system $S=\{X_\alpha, p^{\beta}_\alpha, \tau\}$ whose limit is $X$.
Observe also that each $X_\alpha$ is of weight $<\tau$ because
$p_\alpha(\mathcal U(\alpha))$ is a base for $X_\alpha$ (see
condition (10)).

\smallskip
{\em Claim $3.$ Each $X_\alpha$ is $\mathrm{I}$-favorable with
respect to co-zero sets.}
\smallskip

Indeed, by conditions (9)-(10), $\mathcal B_\alpha=p_\alpha(\mathcal
U(\alpha))$ is a special base for $X_\alpha$ consisting of co-zero
sets. We define a function $\sigma_\alpha:\bigcup\{\mathcal
B_\alpha^n:n\geq 0\}\to\mathcal B_\alpha$ by
$$\sigma_\alpha(p_\alpha\big(U_0),p_\alpha(U_1),..,p_\alpha(U_n)\big)=p_\alpha\big(\sigma(U_0,U_1,..,U_n)\big).$$
This definition is correct because of conditions (9) and (12).
Condition (9) implies that $\sigma_\alpha$ satisfies the hypotheses
of Proposition 2.5. Hence, according to this proposition, $X_\alpha$
is skeletally generated. Finally, by Lemma 2.4, $X_\alpha$ is
$\mathrm{I}$-favorable with respect to co-zero sets.

\smallskip
{\em Claim $4.$ All bonding maps $p^\beta_\alpha$ are skeletal.}
\smallskip

It suffices to show that all $p_\alpha$ are skeletal. And this is
really true because each family $\mathcal U(\alpha)$ is stable with
respect to $\sigma$, see (12). Hence, by \cite[Lemma 9]{kp1}, for
every open set $V\subset X$ there exists $W\in\mathcal U(\alpha)$
such that whenever $U\subset W$ and $U\in\mathcal U(\alpha)$ we have
$V\cap U\neq\varnothing$. The last statement yields that $p_\alpha$
is skeletal. Indeed, let $V\subset X$ be open, and $W\in\mathcal
U(\alpha)$ be as above. Then $p_\alpha(W)$ is a co-zero set in
$X_\alpha$ because of condition (9). We claim that
$p_\alpha(W)\subset\overline{p_\alpha(V)}$. Otherwise,
$p_\alpha(W)\backslash\overline{p_\alpha(V)}$ would be a non-empty
open subset of $X_\alpha$. So, $p_\alpha(U)\subset
p_\alpha(W)\backslash\overline{p_\alpha(V)}$ for some $U\in\mathcal
U(\alpha)$ (recall that $p_\alpha(\mathcal U(\alpha))$ is a base for
$X_\alpha$). Since, by (9), $p_\alpha^{-1}(p_\alpha(U))=U$ and
$p_\alpha^{-1}(p_\alpha(W))=W$, we obtain $U\subset W$ and $U\cap
V=\varnothing$ which is a contradiction.
\end{proof}


\section{Proof of Theorem 1.1 and Corollaries 1.2 - 1.3}

Suppose $X=\displaystyle\mathrm{a}-\underleftarrow{\lim}S$ with
$S=\{X_\alpha, p^{\beta}_\alpha, \alpha<\beta<\tau\}$ being almost
continuous, and $H\subset X$. The set
$$q(H)=\{\alpha:\mathrm{Int}\big(\big((p^{\alpha+1}_{\alpha})^{-1}(\overline{p_{\alpha}(H)})\big)
\backslash\overline{p_{\alpha+1}(H)}\big)\neq\varnothing\}$$ is
called a {\em rank of $H$}.

\begin{lem}
Let $X=\displaystyle\mathrm{a}-\underleftarrow{\lim}S$ and $U\subset
X$ be open, where $S=\{X_\alpha, p^{\beta}_\alpha,
\alpha<\beta<\tau\}$ is almost continuous with skeletal bonding
maps. Then we have:
\begin{itemize}
\item[(i)] $\alpha\not\in q(U)$ if and only if
$(p_\alpha^{\alpha+1})^{-1}\big(\mathrm{Int}\overline{p_\alpha(U)}\big)\subset\overline{p_{\alpha+1}(U)}$;
\item[(ii)] $q(U)\cap[\alpha,\tau)=\varnothing$ provided
$U=p_\alpha^{-1}(V)$ for some open $V\subset X_\alpha$.
\end{itemize}
\end{lem}

\begin{proof}
The first item follows directly from the definition of $q(U)$. For
the second one, suppose $\beta\in q(U)$ for some $\beta\geq\alpha$.
Then
$W=(p_\beta^{\beta+1})^{-1}\big(\mathrm{Int}\overline{p_\beta(U)}\big)\backslash\overline{p_{\beta+1}(U)}\neq\varnothing$
is open in $X_{\beta+1}$. Since $p_\beta^{\beta+1}$ is skeletal,
$\mathrm{Int}\overline{p_\beta^{\beta+1}(W)}$ is a non-empty open
subset of $X_\beta$ which is contained in $\overline{p_\beta(U)}$.
Observe that $p_\beta(U)$ is open in $X_\beta$ because
$p_\beta(U)=(p_\beta^{\alpha})^{-1}(V)$. Hence, $p_\beta(U)\cap
p_\beta^{\beta+1}(W)\neq\varnothing$. The last relation implies
$W\cap p_{\beta+1}(U)\neq\varnothing$ since
$p_{\beta+1}(U)=\big(p^{\beta+1}_\alpha\big)^{-1}(V)=\big(p^{\beta+1}_\alpha\big)^{-1}\big(p_\beta(U)\big)$.
On the other hand, $W\cap p_{\beta+1}(U)=\varnothing$, a
contradiction.
\end{proof}

\begin{lem}
Let $S=\{X_\alpha, p^{\beta}_\alpha, 1\leq\alpha<\beta<\tau\}$ be an
inverse system with skeletal bonding maps and
$X=\displaystyle\underleftarrow{\lim}S$. Suppose $U\subset X$ is
open such that
$(p^\alpha_1)^{-1}\big(\mathrm{Int}\overline{p_1(U)}\big)\subset\mathrm{Int}\overline{p_{\alpha}(U)}$
for all $\alpha<\tau$. Then
$p_1^{-1}\big(\mathrm{Int}\overline{p_1(U)}\big)\subset\overline{U}$.
\end{lem}

\begin{proof}
Suppose
$W=p_1^{-1}\big(\mathrm{Int}\overline{p_1(U)}\big)\backslash\overline{U}\neq\varnothing$.
Then there exists $\mu<\tau$ and open $V\subset X_\mu$ with
$p_\mu^{-1}(V)\subset W$. Hence
$p^\mu_1(V)\subset\mathrm{Int}\overline{p_1(U)}$, so $V\subset
(p^\mu_1)^{-1}\big(\mathrm{Int}\overline{p_1(U)}\big)\subset\mathrm{Int}\overline{p_\mu(U)}$.
The last inclusion implies that $p_\mu^{-1}(V)$ meets
$\overline{p_\alpha(U)}$, a contradiction.
\end{proof}

\begin{lem}
Let $S=\{X_\alpha, p^{\beta}_\alpha, \alpha<\beta<\tau\}$ be a
continuous inverse system with skeletal bonding maps and
$X=\displaystyle\underleftarrow{\lim}S$. Assume $U, V\subset X$ are
open with $q(U)$ and $q(V)$ finite and
$\overline{U}\cap\overline{V}=\varnothing$. If $q(U)\cap q(V)\cap
[\gamma,\tau)=\varnothing$ for some $\gamma<\tau$, then
$\mathrm{Int}\overline{p_\gamma (U)}$ and
$\mathrm{Int}\overline{p_\gamma (V)}$ are disjoint.
\end{lem}

\begin{proof}
Suppose $\mathrm{Int} \overline{p_\gamma (U)}\cap
 \mathrm{Int}\overline{p_\gamma (V)}\neq\varnothing$. We are going to show by transfinite induction that $\mathrm{Int}
\overline{p_\beta(U)}\cap\mathrm{Int} \overline{p_\beta
(V)}\neq\varnothing$ for all $\beta\geq\gamma$. Assume this is done
for all $\beta\in(\gamma, \alpha)$
 with
 $\alpha<\tau$. If $\alpha$ is not a limit cardinal, then $\alpha-1$ belongs to at least one of the
 sets $q(U)$ and $q(V)$. Suppose $\alpha-1\not\in q(V)$.
 Hence,
 $(p^\alpha_{\alpha-1})^{-1}\big(\mathrm{Int}\overline{p_{\alpha-1}(V)}\big)\subset\mathrm{Int}\overline{p_{\alpha}(V)}$
 (see Lemma 3.1(i)). Because of our assumption, $\mathrm{Int}
\overline{p_{\alpha-1}(U)}\cap\mathrm{Int} \overline{p_{\alpha-1}
(V)}\neq\varnothing$. Moreover,
$p^\alpha_{\alpha-1}\big(\overline{p_\alpha(U)}\big)$ is dense in
$\overline{p_{\alpha-1}(U)}$. Hence, $\mathrm{Int}
\overline{p_{\alpha-1}(V)}$ meets
$p^\alpha_{\alpha-1}\big(\overline{p_\alpha(U)}\big)$. This yields
$\mathrm{Int}\overline{p_{\alpha}(V)}\cap\overline{p_\alpha(U)}\neq\varnothing$.
Finally, since by Lemma 2.2(ii) $\overline{p_\alpha(U)}$ is the
closure of its interior,
$\mathrm{Int}\overline{p_{\alpha}(V)}\cap\mathrm{Int}\overline{p_\alpha(U)}\neq\varnothing$.

Suppose $\alpha>\gamma$ is a limit cardinal. Since $q(U)\cap q(V)$
is a finite set, there exists $\lambda\in(\gamma,\alpha)$ such that
$\beta\not\in q(U)\cap q(V)$ for every $\beta\in[\lambda,\alpha)$.
Then for all $\beta\in[\lambda,\alpha)$ we have
$(p_\beta^{\beta+1})^{-1}\big(\mathrm{Int}\overline{p_\beta(U)}\big)\subset\mathrm{Int}\overline{p_{\beta+1}(U)}$
and
$(p_\beta^{\beta+1})^{-1}\big(\mathrm{Int}\overline{p_\beta(V)}\big)\subset\mathrm{Int}\overline{p_{\beta+1}(V)}$.
This allows us to find points
$x_\beta\in\mathrm{Int}\overline{p_\beta(U)}\cap\mathrm{Int}\overline{p_\beta(V)}$,
$\beta\in[\lambda,\alpha)$, such that
$p^\beta_\theta(x_\beta)=x_\theta$ for all
$\lambda\leq\theta\leq\beta<\alpha$. Because $X_\alpha$ is the limit
space of the inverse system
$S^\alpha_\lambda=\{X_\theta,p^\beta_\theta,\lambda\leq\theta\leq\beta<\alpha\}$,
we obtain a point $x_\alpha\in X_\alpha$ with
$p^\alpha_\theta(x_\alpha)=x_\theta$, $\theta\in[\gamma,\alpha)$.
Next claim implies
$x_\alpha\in\mathrm{Int}\overline{p_\alpha(U)}\cap\mathrm{Int}\overline{p_\alpha(V)}$
which completes the induction.

\smallskip
{\em Claim $5$. For all $\theta\in[\lambda,\alpha)$ we have
$(p^\alpha_\theta)^{-1}\big(\mathrm{Int}\overline{p_\theta(V)}\big)\subset\mathrm{Int}\overline{p_{\alpha}(V)}$
and
$(p^\alpha_\theta)^{-1}\big(\mathrm{Int}\overline{p_\theta(U)}\big)\subset\mathrm{Int}\overline{p_{\alpha}(U)}$}.
\smallskip

Fix $\theta\in[\lambda,\alpha)$ and let $\Lambda$ be the set of all
$\beta\in[\theta,\alpha)$ such that
$(p^\beta_\theta)^{-1}\big(\mathrm{Int}\overline{p_\theta(U)}\big)\backslash\overline{p_{\beta}(U)}\neq\varnothing$.
Suppose that $\Lambda\neq\varnothing$ and denote by $\nu$ the
minimal element of $\Lambda$. Therefore
$W_\nu=(p^\nu_\theta)^{-1}\big(\mathrm{Int}\overline{p_\theta(U)}\big)\backslash\overline{p_{\nu}(U)}\neq\varnothing$.
Observe that $\nu>\theta$ because $\theta\not\in q(U)$. Moreover,
$\nu$ is a limit cardinal. Indeed, otherwise
$(p^{\nu-1}_\theta)^{-1}\big(\mathrm{Int}\overline{p_\theta(U)}\big)\subset\mathrm{Int}\overline{p_{\nu-1}(U)}$.
On the other hand $\nu-1\not\in q(U)$ yields
$(p^{\nu}_{\nu-1})^{-1}\big(\mathrm{Int}\overline{p_{\nu-1}(U)}\big)\subset\mathrm{Int}\overline{p_\nu(U)}$.
Hence,
$(p^{\nu}_\theta)^{-1}\big(\mathrm{Int}\overline{p_\theta(U)}\big)\subset\mathrm{Int}\overline{p_{\nu}(U)}$,
a contradiction. So, $X_\nu$ is the limit of the inverse system
$S^{\nu}_\theta=\{X_\beta,p^{\mu}_\beta,\theta\leq\beta\leq\mu<\nu\}$.
Now, we apply Lemma 3.2 to the system $S_{\nu}$ and the set
$\mathrm{Int}\overline{p_\nu(U)}$, to conclude that
$(p^\nu_\theta)^{-1}\big(\mathrm{Int}\overline{p_\theta(U)}\big)\subset\overline{p_{\nu}(U)}$
which contradicts $W_\nu\neq\varnothing$. Consequently,
$\Lambda=\varnothing$ and
$(p^\beta_\theta)^{-1}\big(\mathrm{Int}\overline{p_\theta(U)}\big)\subset\overline{p_{\beta}(U)}$
for all $\beta\in[\theta,\alpha)$. We can apply again Lemma 3.2 to
the system
$S^\alpha_\theta=\{X_\mu,p^\beta_\mu,\theta\leq\mu\leq\beta<\alpha\}$
and the set $\mathrm{Int}\overline{p_\alpha(U)}$ to obtain that
$(p^\alpha_\theta)^{-1}\big(\mathrm{Int}\overline{p_\theta(U)}\big)\subset\mathrm{Int}\overline{p_{\alpha}(U)}$.
Similarly, we can show that
$(p^\alpha_\theta)^{-1}\big(\mathrm{Int}\overline{p_\theta(V)}\big)\subset\mathrm{Int}\overline{p_{\alpha}(V)}$
which completes the proof of Claim 5.

Therefore,
$\mathrm{Int}\overline{p_\beta(U)}\cap\mathrm{Int}\overline{p_\beta(V)}\neq\varnothing$
for all $\beta\in [\gamma,\tau)$. To finish the proof of this lemma,
take $\lambda(0)\in (\gamma,\tau$ such that $\big(q(U)\cup
q(V)\big)\cap [\lambda(0),\tau)=\varnothing$. Repeating the
arguments from Claim 5, we can show that
$(p^\alpha_{\lambda(0)})^{-1}\big(\mathrm{Int}\overline{p_{\lambda(0)}(U)}\big)\subset\mathrm{Int}\overline{p_{\alpha}(U)}$
and
$(p^\alpha_{\lambda(0)})^{-1}\big(\mathrm{Int}\overline{p_{\lambda(0)}(V)}\big)\subset\mathrm{Int}\overline{p_{\alpha}(V)}$
for all $\alpha\in [\lambda(0),\tau)$. Then apply Lemma 3.2 to the
inverse system
$S_{\lambda(0)}=\{X_\mu,p^\beta_\mu,\lambda(0)\leq\mu\leq\beta<\tau\}$
and the set $U$ to obtain that
$p_{\lambda(0)}^{-1}\big(\mathrm{Int}\overline{p_{\lambda(0)}(U)}\big)\subset\mathrm{Int}\overline{U}$.
Similarly, we have
$p_{\lambda(0)}^{-1}\big(\mathrm{Int}\overline{p_{\lambda(0)}(V)}\big)\subset\mathrm{Int}\overline{V}$.
Since
$\mathrm{Int}\overline{p_{\lambda(0)}(U)}\cap\mathrm{Int}\overline{p_{\lambda(0)}(V)}\neq\varnothing$,
the last two inclusions imply
$\overline{U}\cap\overline{V}\neq\varnothing$, a contradiction.
Hence, $\mathrm{Int} \overline{p_\gamma (U)}\cap
 \mathrm{Int}\overline{p_\gamma (V)}=\varnothing$.
\end{proof}

Next proposition was announce in \cite{v}:

\begin{pro}\cite[Proposition 3.2]{v} Let $S=\{X_\alpha,
p^{\beta}_\alpha, \alpha<\beta<\tau\}$ be an almost continuous
inverse system with skeletal bonding maps such that
$X=\displaystyle\mathrm{a}-\underleftarrow{\lim}S$. Then the family
of all open subsets of $X$ having a finite rank is a $\pi$-base for
$X$.
\end{pro}

\begin{proof}
First, following the proof of \cite[Section 3, Lemma 2]{sc1}, we are
going to show by transfinite induction that for every $\alpha<\tau$
the open subsets $U\subset X$ with $q(U)\cap[1,\alpha]$ being finite
form a $\pi$-base for $X$. Obviously, this is true for finite
$\alpha$, and it holds for $\alpha+1$ provided it is true for
$\alpha$. So, it remains to prove this statement for a limit
cardinal $\alpha$ if it is true for any $\beta<\alpha$. Suppose
$G\subset X$ is open. Let $S_\alpha=\{X_\gamma, p^{\beta}_\gamma,
\gamma<\beta<\alpha\}$, $Y_\alpha=\displaystyle\underleftarrow{\lim}
S_\alpha$ and $\tilde{p}^{\alpha}_\gamma\colon Y_\alpha\to X_\gamma$
are the limit projections of $S_\alpha$. Obviously, $X_\alpha$ is
naturally embedded as a dense subset of $Y_\alpha$ and each
$\tilde{p}^{\alpha}_\gamma$ restricted on $X_\alpha$ is
$p^{\alpha}_\gamma$. Then, by Lemma 2.3,
$\mathrm{Int}\overline{p_\alpha(G)}$ is non-empty and open in
$X_\alpha$ (here both interior and closure are taken in $X_\alpha$).
So, there exists $\gamma<\alpha$ and an open set $U_\gamma\subset
X_\gamma$ with $\displaystyle
(\tilde{p}^{\alpha}_\gamma)^{-1}(U_\gamma)\subset\mathrm{Int}_{Y_\alpha}\overline{p_\alpha(G)}^{Y_\alpha}$.
Consequently, $\displaystyle
(p^{\alpha}_\gamma)^{-1}(U_\gamma)\subset\mathrm{Int}\overline{p_\alpha(G)}$.
We can suppose that $U_\gamma=\mathrm{Int}\overline{U_\gamma}$.
Then, according to the inductive assumption,
$p_\gamma^{-1}(U_\gamma)\cap G$ contains an open set $W\subset X$
such that $q(W)\cap [1,\gamma]$ is finite. So,
$W_\gamma=\mathrm{Int}\overline{p_\gamma(W)}\neq\varnothing$ and it
is contained in $U_\gamma$. Hence, $p_\gamma^{-1}(W_\gamma)\cap G$
is a non-empty open subset of $X$ contained in $G$.

\smallskip
{\em Claim $6$. $q\big(p_\gamma^{-1}(W_\gamma)\cap G\big)\cap
[1,\alpha]=q(W)\cap [1,\gamma)$.}
\smallskip

Indeed, for every $\beta\leq\gamma$ we have
$\overline{p_\beta\big(p_\gamma^{-1}(W_\gamma)\cap
G\big)}=\overline{p_\beta(W)}$. This implies
$$q(W)\cap
[1,\gamma)=q\big(p_\gamma^{-1}(W_\gamma)\cap G\big)\cap
[1,\gamma).\leqno{(14)}$$ Moreover, if $\beta\in [\gamma,\alpha)$,
then $$\overline{p_\beta\big(p_\gamma^{-1}(W_\gamma)\cap
G\big)}=\overline{p_\beta\big(p_\gamma^{-1}(W_\gamma)\big)}$$
because $W_\gamma\subset U_\gamma$ and $\displaystyle
(p^{\alpha}_\gamma)^{-1}(U_\gamma)\subset\overline{p_\alpha(G)}$.
Hence, $$q\big(p_\gamma^{-1}(W_\gamma)\cap G\big)\cap
[\gamma,\alpha)=q\big(p_\gamma^{-1}(W_\gamma)\big)\cap
[\gamma,\alpha).\leqno{(15)}$$ Obviously, by Lemma 3.1(ii),
$q\big(p_\gamma^{-1}(W_\gamma)\big)\cap
[\gamma,\alpha)=\varnothing$. Then the combination of $(14)$ and
$(15)$ provides the proof of the claim.

Therefore, for every $\alpha<\tau$ the open sets $W\subset X$ with
$q(W)\cap [1,\alpha]$ finite form a $\pi$-base for $X$. Now, we can
finish the proof of the proposition. If $V\subset X$ is open we find
a set $G\subset V$ with $G=p_\beta^{-1}(G_\beta)$, where $G_\beta$
is open in $X_\beta$. Then there exists an open set $W\subset G$
such that $q(W)\cap [1,\beta]$ is finite. Let
$W_\beta=\mathrm{Int}\overline{p_\beta(W)}$ and
$U=p_\beta^{-1}(W_\beta\cap G_\beta)$. It is easily seen that
$\overline{p_\nu(U)}=\overline{p_\nu(W)}$ for all $\nu\leq\beta$.
This yields that $q(U)\cap [1,\beta)=q(W)\cap [1,\beta)$. On the
other hand, by Lemma 3.1(ii), $q(U)\cap [\beta,\tau)=\varnothing$.
Hence $q(U)$ is finite.
\end{proof}

\begin{pro}
Let $X$ be a compact $\mathrm{I}$-favorable space with respect to
co-zero sets. Then every embedding of $X$ in another space is
$\pi$-regular.
\end{pro}

\begin{proof}
We are going to prove this proposition by transfinite induction with
respect to the weight $w(X)$. This is true if $X$ is metrizable, see
for example \cite[\S21, XI, Theorem 2]{k}. Assume the proposition is
true for any compact space $Y$ of weight $<\tau$  such that $Y$ is
$\mathrm{I}$-favorable with respect to co-zero sets, where $\tau$ is
an uncountable cardinal. Suppose $X$ is compact
$\mathrm{I}$-favorable with respect to co-zero sets and $w(X)=\tau$.
Then, by Theorem 2.6, $X$ is the limit space of a continuous inverse
system $\displaystyle S=\{X_\alpha, p^{\beta}_\alpha,
\alpha<\beta<\tau\}$ such that all $X_\alpha$ are compact
$\mathrm{I}$-favorable with respect to co-zero sets spaces of weight
$<\tau$ and all bonding maps are surjective and skeletal. If
suffices to show that there exists a $\pi$-regular embedding of $X$
in a Tychonoff cube $\I^A$ for some $\mathrm{card}(A)$.

By Proposition 3.4, $X$ has a $\pi$-base $\mathcal B$ consisting of
open sets $U\subset X$  with finite rank. For every $U\in\mathcal B$
let $\Omega(U)=\{\alpha_0, \alpha, \alpha+1:\alpha\in q(U)\}$, where
$\alpha_0<\tau$ is fixed. Obviously, $X$ is a subset of
$\prod\{X_\alpha:\alpha< \tau\}$. For every $U\in\mathcal B$ we
consider the open set $\Gamma(U)\subset\prod\{X_\alpha:\alpha<
\tau\}$ defined by
$$\Gamma(U)=\prod\{\mathrm{Int}\overline{p_\alpha(U)}:\alpha\in
\Omega(U)\}\times\prod\{X_\alpha:\alpha\not\in\Omega(U)\}.$$

\smallskip
{\em Claim $7$. $\Gamma (U_1)\cap\Gamma (U_2)=\varnothing$ whenever
$\overline{U_1}\cap\overline{U_2}=\varnothing$. Moreover, there
exists $\beta\in\Omega(U_1)\cap\Omega(U_2)$ with
$\overline{p_\beta(U_1)}\cap
 \overline{p_\beta(U_2)}=\varnothing$.}

\smallskip
Let $\beta=\max\{\Omega(U_1)\cap \Omega(U_2)\}$. Then
 $\beta$ is either $\alpha_0$ or $\max\{q(U_1)\cap q(U_2)\}+1$. In both
 cases $q(U_1)\cap q(U_2)\cap [\beta,\tau)=\varnothing$.
 According to Lemma 3.3, $\mathrm{Int}\overline{p_\beta(U_1)}\cap
 \mathrm{Int}\overline{p_\beta(U_2)}=\varnothing$.
 Since $\beta\in \Omega(U_1)\cap \Omega(U_2)$,
 $\Gamma(U_1)\cap\Gamma(U_2)=\varnothing$.

Suppose $U\subset X$ is open. Since all $p_\alpha$ and
$p^\beta_\alpha$ are closed skeletal maps (see Lemma 2.2 and Lemma
2.3), $U_\alpha=\mathrm{Int}p_\alpha(U)$ is a non-empty subset of
$X_\alpha$ for every $\alpha$.

\smallskip
{\em Claim $8$. $\bigcap\{ p_\alpha^{-1}(U_\alpha)\cap
U:\alpha\in\Delta\}\neq\varnothing$ for every finite set
$\Delta\subset\{\alpha:\alpha<\tau\}$. }

Obviously, this is true if $|\Delta|=1$. Suppose it is true for all
$\Delta$ with $|\Delta|\leq n$ for some $n$, and let
$\{\alpha_1,..,\alpha_n, \alpha_{n+1}\}$ be a finite set of $n+1$
cardinals $<\tau$. Then $\displaystyle V=\bigcap_{i\leq n}
p_{\alpha_i}^{-1}(U_{\alpha_i})\cap U\neq\varnothing$. Since
$p_{\alpha_{n+1}}$ is skeletal, $\displaystyle
W=\mathrm{Int}p_{\alpha_{n+1}}(V)$ is a non-empty subset of
$\displaystyle X_{\alpha_{n+1}}$, so $W\subset U_{\alpha_{n+1}}$.
Consequently $\displaystyle \bigcap_{i\leq n+1}
p_{\alpha_i}^{-1}(U_{\alpha_i})\cap U\neq\varnothing$.

\smallskip
{\em Claim $9$. $\Gamma (U)\cap X$ is a non-empty subset of
$\overline{U}$ for all $U\in\mathcal B$.}

\smallskip
We are going to show first that $\Gamma (U)\cap X\neq\varnothing$
for all $U\in\mathcal B$. Indeed, we fix such $U$  and let
$\Omega(U)=\{\alpha_i:i\leq k\}$ with $\alpha_i\leq\alpha_j$ for
$i\leq j$. By Claim 8, there exists $\displaystyle x\in
\bigcap_{i\leq k} p_{\alpha_i}^{-1}(U_{\alpha_i})\cap U$. So,
$p_{\alpha_i}(x)\in U_{\alpha_i}$ for all $i\leq k$. This implies
$x\in\Gamma (U)\cap X$.

To show that $\Gamma(U)\cap X\subset\overline{U}$, let
$x\in\Gamma(U)\cap X$. Define $\beta(U)=\max q(U)+1$. Then
$p_{\beta(U)}(x)\in\mathrm{Int}\overline{p_{\beta(U)}(U)}$.
 Since $\alpha\not\in q(U)$ for all $\alpha\geq\beta(U)$, the arguments from Claim 5
 show that
$\big(p^\alpha_{\beta(U)}\big)^{-1}\big(\mathrm{Int}\overline
{p_{\beta(U)}(U)}\big)\subset \mathrm{Int}\overline{p_\alpha(U)}$
for $\alpha\geq\beta(U)$. Hence, applying Lemma 3.2 to the inverse
system $S_U=\{X_\alpha,
p^\beta_\alpha,\beta(U)\leq\alpha\leq\beta<\tau\}$ and the set $U$,
we obtain $x\in p_{\beta(U)}^{-1}\big(\mathrm{Int}\overline
{p_{\beta(U)}(U)}\big)\subset\overline{U}$. This completes the proof
of Claim 9.

According to our assumption, each $X_\alpha$ is $\pi$-regularly
embedded in $\I^{A(\alpha)}$ for some $A(\alpha)$. So, there exists
a $\pi$-regular operator $\reg_\alpha:\mathcal
T_{X_\alpha}\to\mathcal T_{\I^{A(\alpha)}}$. For every $U\in\mathcal
B$ consider the open set
$\theta_1(U)\subset\prod\{\I^{A(\alpha)}:\alpha<\tau\}$,
$$\theta_1(U)=\prod\{\reg_\alpha\big(\mathrm{Int}\overline{p_\alpha(U)}\big):\alpha\in
\Omega(U)\}\times\prod\{\I^{A(\alpha)}:\alpha\not\in \Omega(U)\}.$$
Now, we define a function $\theta$ from $\mathcal B$ to the topology
of $\prod\{\I^{A(\alpha)}:\alpha <\tau\}$ by
$$\theta(G)=\bigcup\{\theta_1(U):U\in\mathcal
B\hbox{~}\mbox{and}\hbox{~}\overline{U}\subset G\}.$$ Let us show
that $\theta$ is $\pi$-regular. It follows from Claim 7 that
$\theta(G_1)\cap\theta(G_2)=\varnothing$ provided $G_1\cap
G_2=\varnothing$. It is easily seen that $\theta(G)\cap
X=\bigcup\{\Gamma(U)\cap X:U\in\mathcal
B\hbox{~}\mbox{and}\hbox{~}\overline{U}\subset G\}.$ According to
Claim 9, each $\Gamma(U)\cap X$ is a non-empty subset of
$\overline{U}$. Hence, $\theta(G)\cap X$ is a non-empty dense subset
of $G$. So, $X$ is $\pi$-regularly embedded in $\I^A$, where $A$ is
the union of all $A(\alpha)$, $\alpha<\tau$.
\end{proof}

\begin{lem}
Suppose $X=\displaystyle\underleftarrow{\lim}S$, where
$S=\{X_\alpha, p^{\beta}_\alpha, A\}$ is an almost $\sigma$-complete
inverse system with open bonding maps and second countable spaces
$X_\alpha$. Then $X$ is ccc and for every open $U\subset X$ there
exists $\alpha\in A$ such that
$p_\beta^{-1}\big(p_\beta(\overline{U})\big) =\overline{U}$.
Moreover, any continuous function $f$ on $X$ can be represented in
the form $f=g\circ p_\alpha$ for some $\alpha\in A$ and a continuous
function $g$ on $X_\alpha$.
\end{lem}

\begin{proof}
More general statement was announce in \cite{v1}, for the sake of
completeness we provide a proof. Denote by $\mathcal B$ a base of
$X$ consisting of all open sets of the form $p_\beta^{-1}(W_\beta)$,
$\beta\in A$, where $W_\beta\subset X_\beta$ is open. Let $U\subset
X$ be open and $\mathcal B(U)=\{V\in\mathcal B:V\subset U\}$. We
construct by induction an increasing sequence $\{\beta_n\}\subset A$
and countable families  $\mathcal B_n(U)\subset\mathcal B(U)$,
$n\geq 1$, satisfying the following conditions:

\begin{itemize}
\item[$(i)_n$] $\mathcal B_n(U)\subset\mathcal B_{n+1}(U)$ for each $n$;
\item[$(ii)_n$] The family $\displaystyle\{p_{\beta_n}(W): W\in\mathcal B_n(U)\}$ is
dense in $p_{\beta_n}(U)$;
\item[$(iii)_n$] $\displaystyle p_{\beta_{n+1}}^{-1}\big(p_{\beta_{n+1}}(W)\big)=W$
for all $n\geq 1$ and $W\in\mathcal B_n(U)$.
\end{itemize}

Fix an arbitrary $\beta_1\in A$ and choose a countable family
$\mathcal B_1(U)\subset\mathcal B(U)$ such that
$\displaystyle\{p_{\beta_1}(W): W\in\mathcal B_1(U)\}$ is dense in
$p_{\beta_1}(U)$ (this can be done because $X_{\beta_1}$ is second
countable). Suppose $\beta_k$ and $\mathcal B_k(U)$ are already
constructed for all $k\leq n$. The family $\mathcal B_n(U)$ is
countable and for each $W\in\mathcal B_n(U)$ there exists
$\beta_W\in A$ with $\displaystyle
p_{\beta_W}^{-1}\big(p_{\beta_W}(W)\big)=W$. Moreover, $A$ is
$\sigma$-complete. So, we can find $\beta_{n+1}\geq\beta_n$
satisfying item $(iii)_n$. Next, we choose a countable family
$\mathcal B_{n+1}\subset\mathcal B$ containing $\mathcal B_{n}$ and
satisfying condition $(ii)_n$. This completes the induction.
Finally, let $\beta=\sup\{\beta_n:n\geq 1\}$ and $\mathcal
B_0=\bigcup_{n\geq 1}\mathcal B_n$. It is easily seen that
$\displaystyle\{p_{\beta}(W): W\in\mathcal B_0\}$ is dense in
$p_{\beta}(U)$ and $\displaystyle
p_{\beta}^{-1}\big(p_{\beta}(W)\big)=W$ for all $W\in\mathcal B_0$.
Since $p_\beta$ is open, this implies that $\bigcup\mathcal B_0$ is
dense in $U$ and $p_\beta^{-1}\big(p_\beta(\overline{U})\big)
=\overline{U}$.

Suppose now $f\colon X\to\mathbb R$ is a continuous function. Choose
a countable base $\mathcal U$ of $\mathbb R$. For each $U\in\mathcal
U$ there exists $\beta(U)\in A$ such that
$p_{\beta(U)}^{-1}\big(p_{\beta(U)}(\overline{U})\big)=\overline{U}$.
Let $\beta=\sup\{\beta(U):U\in\mathcal U\}$. Then
$p_{\beta}^{-1}\big(p_{\beta}(\overline{U})\big)=\overline{U}$ for
all $U\in\mathcal U$. The last equalities imply that if
$p_\beta(x)=p_\beta(y)$ for some $x,y\in X$, then $f(x)=f(y)$. So,
the function $g\colon X_\beta\to\mathbb R$,
$g(z)=f(p_\beta^{-1}(z))$, is well defined and $f=g\circ p_\beta$.
Finally, since $p_\beta$ is open, $g$ is continuous.
\end{proof}

\begin{pro}
Let $Y$ be a limit space of an almost $\sigma$-complete inverse
system with open bonding maps and second countable spaces. Suppose
$X$ is a $\pi$-regularly $C^*$-embedded subspace of $Y$. Then $X$ is
skeletally generated.
\end{pro}

\begin{proof}
Suppose $Y=\displaystyle\underleftarrow{\lim}S_Y$ and
$\displaystyle\reg\colon\mathcal T_X\to\mathcal T_{Y}$ is a
$\pi$-regular operator, where $S_Y=\{Y_\alpha, \pi^{\beta}_\alpha,
A\}$ is an almost $\sigma$-complete inverse system with open bonding
maps and second countable spaces $Y_\alpha$.  Then the limit
projections $\pi_\alpha\colon Y\to Y_\alpha$ are also open.

Let $\mathcal A_\beta$ be a countable open base for $Y_\beta$. We
say that $\beta\in A$ is {\em $\reg$-admissible} if
$$\pi_\beta^{-1}\big(\pi_\beta\big(\overline{\reg(\pi_\beta^{-1}(V)\cap X)}\big)\big)
=\overline{\reg(\pi_\beta^{-1}(V)\cap X)}\leqno{(16)}$$ for every
$V\in\mathcal A_\beta$. We also denote $X_\beta=\pi_\beta(X)$.

{\em Claim $10$. The map $p_\beta=\pi_\beta|X$ is skeletal for every
$\reg$-admissible $\beta\in A$.}

\smallskip
The proof of this claim is extracted from the proof of \cite[Lemma
9]{sh}. Let $U\subset X$ be open in $X$. Because $\pi_\beta$ is
open, it suffices to show that $\displaystyle\pi_\beta(\reg(U))\cap
X_\beta\subset\overline{\pi_\beta(U)}^{X_\beta}$. Suppose there
exists a point $z\in\pi_\beta(\reg(U))\cap
X_\beta\backslash\overline{\pi_\beta(U)}^{X_\beta}$ and take
$V\in\mathcal A_\beta$ containing $z$ such that
$V\cap\overline{\pi_\beta(U)}=\varnothing$ (here
$\overline{\pi_\beta(U)}$ is the closure in $Y_\beta$). Since
$\beta$ is $\reg$-admissible,
$\pi_\beta^{-1}\big(\pi_\beta\big(\overline{\reg(U_1)}\big)\big)=\overline{\reg(U_1)}$,
where $U_1=\pi_\beta^{-1}(V)\cap X$. Obviously, $U_1\cap
U=\varnothing$ and $\pi_\beta(U_1)=V\cap X_\beta$. Because
$\reg(U_1)\cap X$ is dense in $U_1$, we have
$\overline{\pi_\beta\big(\reg(U_1)\cap
X\big)}=\overline{\pi_\beta(U_1)}=\overline{V\cap X_\beta}$. Since
$\pi_\beta\big(\overline{\reg(U_1)}\big)$ is closed in $Y_\beta$
(recall that $\pi_\beta$ being open is a quotient map),
$z\in\pi_\beta\overline{\reg(U_1)}\cap\pi_\beta(\reg(U))$ which
implies $\overline{\reg(U_1)}\cap\reg(U)\neq\varnothing$. So,
$\reg(U_1)\cap\reg(U)\neq\varnothing$, and consequently, $U\cap
U_1\neq\varnothing$. This contradiction completes the proof of Claim
10.

\smallskip
{\em Claim $11$. Let $\{\beta_n\}_{n\geq 1}$ be an increasing
sequence of elements of $A$ such that each $\beta_{n+1}$ satisfies
the equality $(16)$ with $V\in\mathcal A_{\beta_n}$. Then
$\sup\{\beta_n:n\geq 1\}$ is $\reg$-admissible. In particular, this
is true if all $\beta_n$ are $\reg$-admissible.}

The proof of this claim follows from the definition of
$\reg$-admissible sets.

\smallskip
{\em Claim $12$. For every $\gamma\in A$ there exists an
$\reg$-admissible $\beta$  with $\gamma<\beta$.}

\smallskip
We construct by induction an increasing sequence $\{\beta_n\}_{n\geq
1}$ such that $\beta_1=\gamma$ and $\beta_{n+1}$ satisfies the
equality $(16)$ with $V\in\mathcal A_{\beta_n}$ for all $n\geq 1$.
Suppose $\beta_n$ is already constructed. By Lemma 3.6, for each
$V\in\mathcal A_{\beta_n}$ there exists $\beta(V)\in A$ such that
$\pi_{\beta(V)}^{-1}\big(\pi_{\beta(V)}\big(\overline{\reg(\pi_{\beta(V)}^{-1}(V)\cap
X)}\big)\big)=\overline{\reg(\pi_{\beta(V)}^{-1}(V)\cap X)}$ and
$\beta(V)\geq\beta_n$. Then $\beta_{n+1}=\sup\{\beta(V):V\in\mathcal
A_{\beta_n}\}$ is as desired (to be sure that $\beta_{n+1}$ exists,
we may assume that $\{\beta(V):V\in\mathcal A_{\beta_n}\}$ is an
increasing sequence). Finally, by Claim 11, $\beta=\sup\{\beta_n:
n\geq 1\}$ is $\reg$-admissible.

Now, consider the set $\Lambda\subset A$ consisting of all
$\reg$-admissible $\beta$ with the order inherited from $A$.
According to Claim 12, $\Lambda$ is directed. Claim 11 yields
$\Lambda$ is $\sigma$-complete and, by Claim 10, all $p_\beta$ are
skeletal maps. Hence, the bonding maps $p^\alpha_\beta\colon
X_\alpha\to X_\beta$, where $\beta,\alpha\in\Lambda$ and
$X_\alpha=p_\alpha(X)$, are also skeletal. Moreover, the inverse
system $\displaystyle S_X=\{X_\alpha, p^\beta_\alpha, \Lambda\}$ is
$\sigma$-complete and
$X=\displaystyle\mathrm{a}-\underleftarrow{\lim}S_X$. It remains to
show that the system $S_X$ satisfies condition (7). So, let $f\colon
X\to\mathbb R$ be a bounded continuous function. Next, extend $f$ to
a continuous function $\overline{f}\colon Y\to\mathbb R$ (recall
that $X$ is $C^*$-embedded in $Y$). Since any inverse
$\sigma$-complete system with open projections and second countable
spaces is factorizable (i.e., its limit space satisfies condition
(7)), see Lemma 3.6, there exists $\alpha\in\Lambda$ and a
continuous function $g:X_\alpha\to\mathbb R$ with $f=g\circ
p_\alpha$. Therefore, $X$ is skeletally generated.
\end{proof}

{\em Proof of Theorem $1.1$}. To prove implication $(i)\Rightarrow
(ii)$, suppose $X$ is $\mathrm{I}$-favorable with respect to co-zero
sets and $X$ is $C^*$-embedded in a space $Y$. Then
$\overline{X}^{\beta Y}$ is homeomorphic to $\beta X$. Since $\beta
X$ is also $\mathrm{I}$-favorable with respect to co-zero sets (see
Proposition 2.1), according to Proposition 3.5, $\beta X$ is
$\pi$-regularly embedded in $\beta Y$. This yields that $X$ is
$\pi$-regularly embedded in $Y$.

$(ii)\Rightarrow (iii)$ Let $X$ be a $C^*$-embedded subset of some
$\I^A$. Then $X$ is $\pi$-regularly embedded in $\I^A$. Since $\I^A$
is openly generated (it is the limit space of the continuous inverse
system $\{\I^B, \pi^C_B, B\subset C\subset A\}$ with all $B,C$ being
countable subsets of $A$), we can apply Proposition 3.7 to conclude
that $X$ is skeletally generated.

Finally, the implication $(iii)\Rightarrow (i)$ follows from Lemma
2.4 \hfill{$\Box$}

\smallskip
{\em Proof of Corollary $1.2$.} Let $X_\alpha$, $\alpha\in\Lambda$,
be a family of compact $\mathrm I$-favorable with respect to co-zero
sets spaces and $X=\prod_{\alpha\in\Lambda} X_\alpha$. We embed each
$X_\alpha$ is a Tychonoff cube $\mathbb I^{A(\alpha)}$ and let
$K=\prod_{\alpha\in\Lambda}\mathbb I^{A(\alpha)}$. By theorem
1.1(ii), there exists a $\pi$-regular operator $e_\alpha:\mathcal
T_{X_\alpha}\to\mathcal T_{I^{A(\alpha)}}$ for each
$\alpha\in\Lambda$. Let $\mathcal B$ be the family of all standard
open sets of the form $U=U_{\alpha(1)}\times..\times
U_{\alpha(k)}\times\prod\{X_\alpha:\alpha\neq\alpha_i, i=1,..,k\}$,
where each $U_{\alpha(i)}\subset X_{\alpha(i)}$ is open. For any
such $U\in\mathcal B$ we define\\
$\gamma(U)=e_{\alpha(1)}(U_{\alpha(1)})\times...\times
e_{\alpha(k)}(U_{\alpha(k)})\times\prod\{\mathbb
I^{A(\alpha)}:\alpha\neq\alpha_i, i=1,..,k\}$. Finally, we define a
function $e\colon\mathcal T_X\to\mathcal T_K$ by the equality
$e(W)=\bigcup\{\gamma(U):U\in\mathcal B{~}\mbox{and}{~}U\subset
W\}$. It is easily seen that $e$ is $\pi$-regular. Since $K$ is the
limit space of a continuous $\sigma$-complete inverse system
consisting of open bounding maps and compact metrizable spaces, by
Proposition 3.7, $X$ is skeletally generated. Hence, $X$ is $\mathrm
I$-favorable with respect to co-zero sets. \hfill{$\Box$}

\smallskip
{\em Proof of Corollary $1.3$.} Suppose $X\subset Y$ a
$C^*$-embedded $\mathrm I$-favorable space with respect to co-zero
sets, where $Y$ is extremally disconnected. Then, by Theorem
1.1(ii), there exists a $\pi$-regular operator
$\mathrm{e}\colon\mathcal T_X\to\mathcal T_Y$. We need to show that
the closure (in $X$) of every open subset of $X$ is also open. Since
$Y$ is extremally disconnected, $\overline{\mathrm{e}(U)}^Y$ is open
in $Y$. So, the proof will be done if we prove that
$\overline{\mathrm{e}(U)}^Y\cap X=\overline{U}^X$ for all
$U\in\mathcal T_X$. By $(1)$, we have
$\overline{U}^X\subset\overline{\mathrm{e}(U)}^Y\cap X$. Assume
there exists $x\in\overline{\mathrm{e}(U)}^Y\cap
X\backslash\overline{U}^X$ and choose $V\in\mathcal T_X$ with
$V\subset\overline{\mathrm{e}(U)}^Y\backslash\overline{U}^X$. Then
$\mathrm{e}(V)\cap\overline{\mathrm{e}(U)}^Y\neq\varnothing$, so
$\mathrm{e}(V)\cap\mathrm{e}(U)\neq\varnothing$. The last one
contradicts $U\cap V=\varnothing$. \hfill{$\Box$}

\smallskip
\textbf{Acknowledgments.} The author would like to express his
gratitude to A. Kucharski and S. Plewik for providing their recent
papers on $\mathrm I$-favorable spaces, and for their critical
remarks concerning the first version of this paper. Thanks also go
to Prof. M. Choban for supplying me with the Bereznicki\v{i} results
from \cite{b}.



\begin{thebibliography}{999}























\bibitem{b}
Ju.~Bereznicki\v{i}, \textit{On theory of absolutes}, In: III
Tiraspol Symposium on General Topology and its applications,
Kishinev, Shtiinca, 1973, p. 13--15.

\bibitem{dkz}
P.~Daniels,~K.~Kunen and H.~Zhou \textit{On the open-open game},
Fund. Math.{\textbf 145} (1994), no. 3, 205--220.

\bibitem{kp1}
A.~Kucharski and S.~Plewik, \textit{Inverse systems and
$I$-favorable spaces}, Topology Appl. \textbf{156} (2008), no. 1,
110--116.

\bibitem{kp2}
A.~Kucharski and S.~Plewik, \textit{Game approach to universally
Kuratowski-Ulam spaces}, Topology Appl. \textbf{154} (2007), no. 2,
421--427.

\bibitem{kp3}
A.~Kucharski and S.~Plewik, \textit{Skeletal maps and $I$-favorable
spaces}, arXiv:1003.2308v1 [math.GN] 11 Mar 2010.

\bibitem{k}
K.~Kuratowski, \textit{Topology, vol. I}, Academic Press, New York;
PWN-Polish Scientific Publishers, Warsaw 1966.

\bibitem{s}
L.~Shapiro, \textit{On a spectral representation of images of
$\kappa$-metrizable bicompacta}, Uspehi Mat. Nauk \textbf{37}
(1982), no. 2, 245--246 (in Rusian).

\bibitem{sc1}
E.~Shchepin, \textit{Topology of limit spaces of uncountable inverse
spectra}, Russian Math. Surveys \textbf{315} (1976), 155--191.

\bibitem{sc2}
E.~Shchepin, \textit{$k$-metrizable spaces}, Math. USSR Izves.
\textbf{14} (1980), no. 2, 407--440.

\bibitem{sc3}
E.~Shchepin, \textit{Functors and uncountable degrees of compacta},
Uspekhi Mat. Nauk \textbf{36} (1981), no. 3, 3--62 (in Russian).


\bibitem{sh}
L.~Shapiro, \textit{On spaces co-absolute with a generalized Cantor
discontinuum}, Doklady Akad. Nauk SSSR {\bf 288} (1986), no. 6,
1322--1326 (in Russian).

\bibitem{s}
L.~Shirokov, \textit{An external characterization of Dugundji spaces
and $k$-metrizable compacta}, Dokl. Akad. Nauk SSSR \textbf{263}
(1982), no. 5, 1073--1077 (in Russian).





\bibitem{v}
V.~Valov, \textit{Some characterizations of the spaces with a
lattice of $d$-open mappings}, C. R. Acad. Bulgare Sci \textbf{39}
(1986), no. 9, 9--12.

\bibitem{v1}
V.~Valov, \textit{A note on spaces with a lattice of $d$-open
mappings}, C. R. Acad. Bulgare Sci \textbf{39} (1986), no. 8, 9--12.
\end{thebibliography}
\end{document}